\documentclass{amsart}
\usepackage{hyperref}

\newtheorem{thm}{Theorem}
\newtheorem{prop}[thm]{Proposition}
\newtheorem{lem}[thm]{Lemma}

\theoremstyle{definition}
\newtheorem{ex}[thm]{Example}
\newtheorem{rem}[thm]{Remark}
\newtheorem{question}[thm]{Question}

\newcommand{\Gal}{\operatorname{Gal}}
\newcommand{\Trop}{\operatorname{Trop}}
\newcommand{\val}{\operatorname{val}}
\newcommand{\CC}{\mathbb C}
\newcommand{\GG}{\mathbb G}

\newcommand{\RR}{\mathbb R}
\newcommand{\ZZ}{\mathbb Z}

\begin{document}

\title{Connectivity of tropicalizations}
\author{Dustin Cartwright \and Sam Payne}

\dedicatory{Dedicated to Bernd Sturmfels on the~occasion of
his~fiftieth~birthday}

\begin{abstract}
We show that the tropicalization of an irreducible variety over a complete or algebraically
closed valued field is connected through codimension~1, giving an affirmative
answer in all characteristics to a question posed by Einsiedler, Lind, and Thomas in 2003.
\end{abstract}
\maketitle

\vspace{-20pt}

\section{Introduction}

Let $X$ be a closed subvariety of the torus $\GG_m^n$ over a valued field $K$. The tropicalization $\Trop(X)$ is the image in $\RR^n$ of the nonarchimedean analytification of~$X$ under the coordinate-wise valuation map. It is the support of a polyhedral complex of pure dimension equal to the algebraic dimension of $X$.  

We say that a pure-dimensional polyhedral complex is \emph{connected through codimension 1} if, for any pair of facets $F$ and~$F'$, there exists a sequence of facets $F = F_0, \ldots, F_s = F'$ such that $F_i \cap F_{i-1}$ is a face of codimension~1, for $1 \leq i \leq s$. This property is independent of the choice of polyhedral structure because a connected pure-dimensional polyhedral complex is connected through codimension 1 if and only if it cannot be disconnected by removing a finite union of polyhedra of codimension at least 2.  Such complexes are also sometimes called strongly connected \cite[11.6]{Bjorner95}.

\begin{thm} \label{thm:main}
Suppose $X$ is irreducible and $K$ is either algebraically closed, complete, or real closed with convex valuation ring.  Then $\Trop(X)$ is connected through codimension 1.
\end{thm}

\noindent The theorem applies, in particular, to irreducible varieties over any field with the trivial valuation, since the trivial valuation is always complete.  

Theorem~\ref{thm:main} says that one can traverse all facets of the tropicalization of an irreducible variety over a suitable field by stepping from facet to facet across codimension 1 faces.  This property is crucial for computational purposes, because such stepping procedures are at the heart of standard algorithms for computing tropicalizations, such as those implemented in the softward package Gfan~\cite{gfan}.  

Connectedness through codimension 1 for tropicalizations of irreducible varieties over algebraically closed nonarchimedean fields was posed as a question by Einsiedler, Lind, and Thomas during a problem session at the 2003 AIM workshop on ``Amoebas and tropical geometry'' \cite[Problem~A.6]{AIM-Tropical}.  A solution is presented over the complex numbers with respect to the trivial valuation in \cite{BJSST}, but the argument is flawed, as described in Remark~\ref{rem:error}.  Theorem~\ref{thm:main} gives an affirmative answer to the original question posed at AIM.

\bigskip

Although the question of connectedness through codimension 1 was originally stated over algebraically closed fields, we include the cases of complete fields and real closed fields with convex valuation ring because there has been significant recent interest in extending the basic results of tropical geometry to non-closed fields, especially through connections to nonarchimedean analysis  \cite{Gubler12}, and these cases are both natural and readily reduced to the algebraically closed case (Proposition~\ref{prop:component}).  Completeness is the usual hypothesis for applying analytic techniques, while convexity of the valuation ring in a real closed field is equivalent to several other natural compatibility conditions connecting the order relation to the valuation \cite[Proposition~2.2.4]{EnglerPrestel05} and is used for model theory and elimination of quantifiers with positivity \cite{Prestel07}.

If $K$ does not satisfy any of the hypotheses of Theorem~\ref{thm:main}, then the tropicalization of an irreducible variety over $K$ may be disconnected.  Moreover, even when such a tropicalization is connected, it may fail to be connected through codimension 1.  The tropicalization does not change when one passes to a field extension, but an irreducible variety may become reducible after passing to the completion or algebraic closure, in which case Theorem~\ref{thm:main} does not apply.  The following two examples give irreducible varieties whose tropicalizations are disconnected (Example~\ref{ex:disconnected}) or connected, but not through codimension~1 (Example~\ref{ex:connected-codim-2}).

\begin{ex} \label{ex:disconnected}
Let $p$ be a prime number.  The polynomial $f = x^2 + x + p$ is irreducible over the rational numbers, so its zero locus $X$ is an irreducible subvariety of $\GG_m$.  However, the tropicalization of $X$ with respect to the $p$-adic valuation is disconnected, consisting of the two points $0$ and $1$.
\end{ex}

\begin{ex}\label{ex:connected-codim-2}
Let $K$ be the field of rational functions $\CC(t)$ with the $t$-adic valuation.  Consider the linear space $X$ in $\GG_m^4$ cut out by the linear functions
\[
f_1 = x_1 + x_2 + x_3 + x_4 \mbox{ \ \ and \ \ } f_2 = 1 + x_1 + 2x_2 + 3x_3 + 4x_4.
\]
Its tropicalization $\Delta = \Trop(X)$ is a ``generic tropical plane'' in $\RR^4$, which is the union of the cones spanned by any two of the vectors $e_1, e_2, e_3, e_4, (-e_1 - \cdots - e_4)$.

Now let $Y$ be the variety in $\GG_m^5$ defined by $f_1$, $f_2$, and the irreducible quadratic polynomial $g = x_5^2 + x_5 + t$. Since $Y$ is the product of a geometrically irreducible variety with an irreducible variety, it is irreducible.  Note that $g$ has two distinct roots over $\CC\{\!\{t\}\!\}$ with valuation 0 and 1, and thus, the tropicalization of $Y$ is the disjoint union
\[
\Trop(Y) = \Delta \cup \left( \Delta + e_5 \right),
\]
where we identify $\Delta$ with its image in the hyperplane $\RR^4 \times \{0\}$ in $\RR^5$..

Let $Z$ be the image of $Y$ under the monomial map from $\GG_m^5$ to $\GG_m^4$ that takes $(x_1, x_2, x_3, x_4, x_5)$ to $(x_1x_5, x_2x_5^2, x_3x_5^3, x_4x_5^4)$. Then $Z$ is irreducible and $\Trop(Z)$ is the image of $\Trop(Y)$ under the corresponding linear map from $\RR^5$ to $\RR^4$, which is the union
\[
\Trop(Z) = \Delta \cup \left( \Delta + (1,2,3,4) \right).
\]
Each of $\Delta$ and $\Delta + (1,2,3,4)$ is connected through codimension 1, but their intersection is a single point, namely $(0,0,1,2)$.  Hence $\Trop(Z)$ is connected, but not connected through codimension 1.
\end{ex}

\subsection*{Acknowledgments.}
We thank Bernd Sturmfels and Walter Gubler for suggesting the problem of connectivity of tropicalizations in positive and mixed characteristic, and for helpful comments on an earlier draft of this work.  DC is supported by the National Science Foundation under award DMS-1103856, and SP is partially supported by DMS-1068689.  We are grateful to the Max Planck Institute in Bonn, where this work was done, for its warm hospitality and ideal working conditions.

\section{Proof of connectivity}

The main part of the proof is structured as an induction on dimension, similar to that in \cite{BJSST}.  The base of the induction consists of showing that the tropicalization of an irreducible curve is connected.  The induction is done by intersecting with the tropicalization of a hyperplane in suitable coordinates, using the Bertini irreducibility theorem, and applying the lifting theorem for proper intersections of tropicalizations.  One important difference from the argument in \cite{BJSST} is that we intersect with tropicalizations of hyperplanes in the inductive step, rather than with affine hyperplanes in $\RR^n$. As explained in Remark~\ref{rem:error}, intersections with affine hyperplanes may be disconnected, even when the intersection is proper.

We use the following proposition to reduce Theorem~\ref{thm:main} to the case where $K$ is algebraically closed.

\begin{prop} \label{prop:component}
Suppose $X$ is irreducible and $K$ is either complete or real closed with convex valuation ring. Let $\overline K$ be an algebraic closure of $K$ equipped with an extension of the given valuation on $K$.  Then the tropicalization of any irreducible component of $X_{\overline K}$ is equal to $\Trop(X)$.
\end{prop}

\begin{proof}
First, we claim that if $K$ is complete or real closed with convex valuation ring then the valuation on $K$ extends uniquely to $\overline K$, and hence is invariant under the action of the Galois group $\Gal(\overline K | K)$.  In the case where $K$ is complete and the valuation is nontrivial, this is Proposition~XII.2.5 in~\cite{Lang02}.  Suppose the valuation on $K$ is trivial.
Then the Newton polygon of any polynomial with coefficients in $K$ has only one lower face, with slope zero.   It follows that the only extension of the valuation on $K$ to an algebraic extension field is the trivial one.

Now suppose that $K$ is real closed with convex valuation ring, so its algebraic closure is $\overline K = K(\sqrt{-1})$. Let $\alpha$ be an element of $\overline K \setminus K$, and let $f = u^2 + bu + c$ be its minimal polynomial. Since $f$ is irreducible over $K$, its discriminant $b^2 - 4c$ must be negative. Then, because the valuation ring of $K$ is convex and $b^2$ is positive, we must have
\[
\val(b^2) \geq \val(c).
\]
It follows that the Newton polygon of $f$ has a single lower face, with slope $\val(c)/2$, and hence the valuation of $\alpha$ must be $\val(c)/2$.  In particular, the extension of the valuation to~$\overline K$ is unique, as claimed.

Now, consider the tropicalizations of the irreducible components of $X_{\overline K}$.  Since $X$ is irreducible, the Galois group $\Gal(\overline K | K)$ acts transitively on the irreducible components and, as shown above, the Galois group commutes with the valuation on $\overline K$.  Therefore, for any two irreducible components of $X_{\overline K}$, there is an element of the Galois group that induces a bijection between their analytifications.  Moreover, since the coordinate functions on $\GG_m^n$ are defined over $K$, these bijections commute with the coordinatewise valuation map.  Hence, any two irreducible components have the same tropicalization, which must be equal to $\Trop(X_{\overline K})$.
\end{proof}

For the induction step of our proof of Theorem~\ref{thm:main} we will use a change of coordinates on $\ZZ^n$ followed by an intersection with the tropicalization of a hyperplane.  Let $a_1, \ldots, a_n$
be elements of $K^*$, let $v$ be the vector $(\val(a_1), \ldots, \val(a_n))$, and let $\Delta$
be the $(n-1)$-dimensional fan in $\RR^n$ whose maximal cones are spanned by any
$n-1$ of the vectors $e_1, \ldots, e_n, (-e_1 - \ldots -e_n)$.  Then the
tropicalization of the hyperplane $H$ cut out by $1+ a_1 x_1 + \cdots + a_n x_n$
is the translation
\[
\Trop(H) = \Delta - v.
\]
Each facet is normal to one of the vectors $e_1, \ldots, e_n$, or $e_j - e_i$ for $1 \leq i < j \leq n$.  We will use the following lemma to choose coordinates so that translates of $\Delta$ intersect $\Trop(X)$ as transversely as possible.  

\begin{lem} \label{lem:coords}
There is a basis $\{f_1, \ldots, f_n\}$ for $\ZZ^n$ such that none of the vectors $f_1, \ldots, f_n$ or $f_j - f_i$ for $1 \leq i < j \leq n$ are perpendicular to a positive dimensional face of $\Trop(X)$.  
\end{lem}

\begin{proof}
To construct such a basis, first choose a primitive vector $f_1$ in $\ZZ^n$ that
is not perpendicular to any positive dimensional face of $\Trop(X)$.  Then,
choose $g_2, \ldots, g_n$ such that $\{f_1, g_2, \ldots, g_n\}$ is a basis for
$\ZZ^n$.  By the choice of $f_1$, for any vector $g \in \ZZ^n$ there are only
finitely many integers $a$ such that $g + af_1$ is perpendicular to a positive
dimensional face of $\Trop(X)$.  Therefore, we can choose an integer $a_2$ such
that $f_2 = g_2 + a_2 f_1$ and $f_2 - f_1$ are not perpendicular to any positive
dimensional face of $\Trop(X)$.  Iterating this process, for $3 \leq j \leq n$,
we choose an integer~$a_j$ such that $f_j = g_j + a_j f_1$ and $f_j - f_i$
for $1 \leq i < j$ are not perpendicular to any positive dimensional face of
$\Trop(X)$.   After choosing $f_n$, we arrive at a basis $\{f_1, \ldots, f_n\}$
with the required property.
\end{proof}

\begin{proof}[Proof of Theorem~\ref{thm:main}]
First, we reduce to the case where $K$ is algebraically closed and complete with respect to a nontrivial valuation.  If $K$ is complete or real closed with a convex valuation ring, 
then, by Proposition~\ref{prop:component}, we can extend scalars to the algebraic closure $\overline K$ and replace $X$ by any irreducible component of $X_{\overline K}$, without changing the tropicalization.

We now assume that $K$ is algebraically closed.  Let $L$ be the completion of
the algebraic closure of an extension of $K$ with nontrivial valuation.  Then
$L$ is algebraically closed \cite[Proposition~3.4.1.3]{BGR84},  $X_L$ is
irreducible \cite[Theorem~4.4.4]{EGA4.2}, and  $\Trop(X_L)$ is equal to
$\Trop(X)$ \cite[Proposition~6.1]{analytification}.  Therefore, after replacing
$X$ by $X_L$, we can assume that $K$ is algebraically closed and complete
with respect to a nontrivial valuation.  Under these assumptions, the analytification of $X$ is connected
\cite[Theorem~3.4.5]{Berkovich90}.  Since
$\Trop(X)$ is the image of the analytification under a continuous map, it follows that $\Trop(X)$ is connected.  This proves the theorem if $X$ has dimension
at most 1.  We now proceed by induction on the dimension of $X$.

Suppose that $X$ has dimension at least 2.  Fix
a polyhedral decomposition of~$\Trop(X)$, and let $F$ and $F'$ be two facets in this decomposition.  We say that two facets in a pure-dimensional polyhedral complex are adjacent if they share a face of codimension 1.  To prove the theorem, we will produce a sequence of facets $F = F_0, \ldots,
F_s = F'$ in $\Trop(X)$ such that $F_{i-1}$ and $F_i$ are either equal or adjacent, for $1 \leq i \leq s$.

By Lemma~\ref{lem:coords}, after a change of coordinates
on $\GG_m^n$, we may assume that no positive dimensional face of $\Trop(X)$ is
contained in any translate of the generic tropical hyperplane $\Delta$.  Since
$K$
is algebraically closed with a nontrivial valuation, the value group is dense in
$\RR$.  Therefore, we can choose a vector $v$ with entries in the value group
such that the translate $\Delta - v$ meets $F$ and $F'$ in their relative interiors, and does not
contain any vertices of $\Trop(X)$. 

Now consider hyperplanes in $\GG_m^n$ whose tropicalization is~$\Delta-v$. Explicitly,
such hyperplanes can be given by the equation $1 + a_1 x_1 + \cdots + a_n x_n$,
where the coefficients~$a_i$ are non-zero elements of $K$ with valuations
$\val(a_i) = v_i$. The set of such coefficients is dense in $\GG_m^n$. By
\cite[Theorem~6.3.4]{Jouanolou83}, there is a dense open set of coefficients
such that the intersection of the corresponding hyperplane with~$X$ is
irreducible. Combining these two statements, we can choose $H$ such that
$\Trop(H) = \Delta - v$ and $H \cap X$ is irreducible.

By the choice of coordinates, $\Trop(H)$ meets each face of $\Trop(X)$ properly.
Thus, it follows
from
\cite[Theorem~1.1]{tropicallifting} that
\[
\Trop(X \cap H) = \Trop(X) \cap \Trop(H).
\]
Furthermore, if we give $\Trop(X) \cap \Trop(H)$ the polyhedral structure coming
from the intersection, then each facet of
 $\Trop(X) \cap \Trop(H)$ is contained in a unique facet of $\Trop(X)$.  Furthermore, if $G$ and $G'$ are adjacent facets of $\Trop(X) \cap \Trop(H)$, then the facets of $\Trop(X)$ that contain them are either adjacent or equal.   Now, let $G$ and $G'$ be facets of $\Trop(X \cap H)$ that are contained in $F$ and
$F'$, respectively. Since $X \cap H$ is irreducible, the induction hypothesis
says that $\Trop(X \cap H)$ is connected through codimension 1, so there is a
sequence of facets $G = G_0, \ldots, G_s = G'$ in $\Trop(X \cap H)$ such that
$G_{i-1}$ is adjacent to $G_i$ for $1 \leq i \leq s$.  Let
$F_i$ be the unique facet of $\Trop(X)$ that contains $G_i$.  Then $F_{i-1}$ and $F_i$ are either equal or adjacent, for $1 \leq i \leq s$, and the theorem follows.
\end{proof}

\begin{rem}
We used the connectedness of analytifications of irreducible curves over complete fields for the base case of the induction in the proof of Theorem~\ref{thm:main}.
It is also possible to give a purely algebraic proof, which we outline here.

Suppose that $X$ is a curve in $\GG_m^n$ over an algebraically closed field $K$ that is complete with respect to a nontrivial valuation.  Choose a tropical compactification $\mathcal X$ of $X$ in a toric scheme over the valuation ring, as in \cite[Section~12]{Gubler12}.  Like any one-dimensional polyhedral complex, $\Trop(X)$ deformation retracts onto the union of its bounded faces, and the union of its bounded faces is the image of the dual graph of the special fiber of $\mathcal X$ under a continuous map.  Therefore, it is enough to show that the special fiber of this tropical compactification is connected.  If this model is defined over the valuation ring in a discretely valued subfield, then the tropical compactification is noetherian, irreducible, and proper over the DVR, and hence its special fiber is connected, by Zariski's connectedness theorem \cite[Section~4.3]{EGA3.1}.  If it is not defined over a DVR, the tropical compactification is still flat and finite type over the valuation ring, by \cite[Proposition~6.7]{Gubler12} and hence finitely presented \cite[Corollary~3.4.7]{RaynaudGruson71}.  Connectedness of the special fiber then follows from Zariski's connectedness theorem after noetherian approximation.
\end{rem}

\begin{rem} \label{rem:error}
The proof of \cite[Theorem~14]{BJSST} is similarly structured as an induction on
the dimension of $X$, but contains errors in the base case and in the
induction step, as we now explain. The base case involves passing from a
surface over $\CC$ to a curve over
the field of Puiseux series $\CC\{\!\{t\}\!\}$, via an extension of scalars, before using the connectedness of analytifications of curves over $\CC\{\!\{t\}\!\}$.  However, this extension of scalars does not preserve irreducibility in general.  Specifically, in the fifth line of the proof, $I$ denotes the ideal defining $X$
in the
coordinate ring of $\GG_m^n$ over $\CC$ and $I'$ is the extension of $I'$ into
the coordinate ring of $\GG_m^{n-1}$ over
$\CC\{\!\{t\}\!\}$ induced by setting $x_n$ equal to $t$. It is claimed that
$I'$ is prime, but this is false in general, as the following example illustrates.

Suppose $I$ is the principal ideal generated by $x_1^2 x_n + x_1 + x_n$. Then
$I$ is prime, but $x_1^2 t + x_1 + t$ is a product of two distinct linear
factors over the field of Puiseux series, and hence neither $I'$ nor its radical
is prime.  This error is critical, since the irreducibility of the scheme $X'$
cut out by $I'$ is used to conclude that $\Trop(X')$ is connected.  In this
counterexample, not only is $X'$ reducible, but $\Trop(X')$ is disconnected.

There is also an error in the induction step; the intersection that is used to reduce dimension does not preserve irreducibility in general. Specifically, in the fourth line from the
bottom
of the proof, it is claimed that the intersection of~$X$ with a generic
translate of a subtorus of codimension 1 is irreducible.  This claim is false in
general, and the hypersurface defined by $x_1^2 x_n + x_1 + x_n$ is again a
counterexample.  The intersection of this hypersurface with a generic translate of the
codimension 1 subtorus $T$ cut out by $x_n = 1$ is reducible.  The irreduciblity
is claimed to follow from the Kleiman-Bertini theorem
\cite[Theorem~III.10.8]{Hartshorne77}, which would say that the intersection
with a general translate by the torus is smooth, provided that $X$ is also
smooth.  In the counterexample above, the hypersurface is smooth, and its
intersection with a generic translate of $T$ is smooth but disconnected.
\end{rem}

The examples in Remark~\ref{rem:error} involve a subvariety that is preserved by a positive dimensional subtorus.  Such subvarieties have many exceptional properties related to tropical geometry.  See, for instance, \cite[Theorem~1.1]{adelicamoebas} and \cite[Theorem~3.1]{HackingKeelTevelev09}.  It should be interesting to investigate under what conditions generic translates of subtori have irreducible intersection with a given subvariety.

\begin{question}
Suppose $X$ is an irreducible subvariety of $\GG_m^n$ over an algebraically closed field.  For which subtori of $\GG_m^n$ is the intersection of a generic translate with $X$ irreducible? 
\end{question}

\noindent The existence of such a subtorus with dimension~$(n-1)$ would imply that the intersection of $\Trop(X)$ with a generic translate of the corresponding hyperplane is proper and connected through codimension~1. 

\begin{question}
Does there always exist an affine hyperplane in $\RR^n$ whose intersection with $\Trop(X)$ is proper and connected through codimension~1, or even just connected?
\end{question}

\end{document}